\documentclass[a4paper,leqno,10pt]{article}

\usepackage{epsf,graphicx,epsfig}
\usepackage{amscd}
\usepackage{amsmath,latexsym,amssymb,amsthm}
\usepackage[nospace,noadjust]{cite}
\usepackage{textcomp}
\usepackage{setspace,cite}
\usepackage{lscape,fancyhdr,fancybox}
\usepackage{textcomp}
\usepackage{bm}
\usepackage{stmaryrd}
\usepackage[all,cmtip]{xy}
\usepackage{tikz}
\usetikzlibrary{shapes,arrows,decorations.markings}
\usepackage{graphicx}

\usepackage{color}
\usepackage{url}
\usepackage{enumerate}
\usepackage[mathscr]{euscript}

  \theoremstyle{plain}

\swapnumbers
    \newtheorem{thm}{Theorem}[section]
    \newtheorem{prop}[thm]{Proposition}
   
    \newtheorem{corollary}[thm]{Corollary}
    
    \newtheorem{subsec}[thm]{}
\theoremstyle{definition}
    \newtheorem{defn}[thm]{Definition}
    \newtheorem{exam}[thm]{Example}

\theoremstyle{remark}
     \newtheorem{remark}[thm]{Remark}

\title{}
\author{}
\date{}
\usepackage{amssymb}

\usepackage{hyperref}
\hypersetup{
	colorlinks,
	citecolor=blue,
	filecolor=black,
	linkcolor=blue,
	urlcolor=black
}

\title{Reduction of Nambu-Poisson manifolds by regular distributions}

\author{Apurba Das \\ Stat-Math Unit,
Indian Statistical Institute, Kolkata 700108,
West Bengal, India.\\
Email: apurbadas348@gmail.com}

\begin{document}

\maketitle
\begin{abstract}
The version of Marsden-Ratiu reduction theorem for Nambu-Poisson manifolds by a regular distribution has been studied by Ib$\acute{\text{a}}\tilde{\text{n}}$ez
et al. In this paper we show that the reduction is always ensured unless the distribution is zero. Next we extend the more general Falceto-Zambon Poisson reduction theorem for Nambu-Poisson manifolds. Finally, we define gauge 
transformations of Nambu-Poisson structures and show that these transformations commute with the reduction procedure.\\
\end{abstract}

\thispagestyle{empty}

\maketitle


\vspace{0.2cm}
\section{Introduction}
The reduction procedure is very useful in dynamical systems, since it gives rise to another system with less degrees of freedom. 
The most general reduction theorem for Hamiltonian systems is the Marsden-Ratiu reduction of Poisson manifolds \cite{mars-ratiu} (see also \cite{ort-rat-book}). Given a Poisson manifold
$M$ and a submanifold $N \subset M$, they considered a canonical vector subbundle $E$ of the tangent bundle $TM$ restricted to $N$ such that 
$E \cap TN$ defines a regular, integrable distribution on $N$ (hence, it defines a regular foliation $\mathcal{F}$ on $N$). The question was the
existence of a Poisson structure on the
quotient $N / \mathcal{F}$ from the one on $M$. In \cite{mars-ratiu}, the authors gave a necessary and sufficient condition to ensure this. 
The Marsden-Ratiu reduction by distributions has been reformulated by Falceto and Zambon in \cite{fal-zam}.
It was shown in \cite{fal-zam} that the sufficient condition for the Marsden-Ratiu reduction theorem always hold unless the distribution
$E$ is zero. Moreover, they observed that the assumptions on the subbundle $E$ and the sufficient condition of the Marsden-Ratiu reduction theorem can be weakend to ensure the Poisson structure on $N / \mathcal{F}$.

In \cite{nambu}, Y. Nambu introduced a generalization of Hamiltonian dynamics which is based on ternary operations. To outline the basic principles of Nambu's generalized dynamics,
L. Takhtajan \cite{takh} introduced the notion of Nambu-Poisson manifolds as $n$-ary generalization of Poisson manifolds. Later, Nambu mechanics and properties of Nambu-Poisson
manifolds were extensively studied by several authors from different perspectives \cite{chat, chat-takh, curt-zachos, dito-flato-stern-takh, gau, iban-leon-mar-pad}.
A Nambu-Poisson manifold of order $n$ is a manifold $M$ equipped with a skew-symmetric $n$-ary bracket on $C^\infty(M)$ which satisfies the Leibniz rule and 
the fundamental identity (generalization of the Jacobi identity). Like Poisson manifolds, a Nambu-Poisson manifold gives rise to a singular foliation on the manifold. Moreover,
a Nambu-Poisson manifold of order $n$ corresponds to a Leibniz algebroid
on the $(n-1)$-th exterior power of its cotangent bundle. A remarkable difference
between a Poisson manifold and a Nambu-Poisson manifold of order greater than $2$ is that in the latter case, the associated (Nambu) tensor is locally decomposable.

Following the Poisson reduction theorem of Marsden and Ratiu, Ib$\acute{\text{a}}\tilde{\text{n}}$ez et al. \cite{iban-leon-mar-dieg} 
considered a similar set-up for Nambu-Poisson manifolds and studied the reduction of Nambu-Poisson manifolds.
More precisely, let $M$ be a Nambu-Poisson manifold, $N \subset M$ a submanifold and $E \subset TM|_N$ be a canonical vector subbundle of $TM$ restricted to
$N$ such that $E \cap TN$ defines a regular, integrable distribution on $N$. Similar to the Poisson case, the authors gave a necessary and sufficient condition
to ensure the existence of a Nambu-Poisson structure on $N / \mathcal{F}$ from the one on $M$. 

The main aim of the present work is to put it in record in the literature that most of the results of \cite{fal-zam} and their proofs extend in a natural way to
the context of Nambu-Poisson manifolds. To do this, we closely follow \cite{fal-zam} to adapt the definitions and methods of the proofs therein to prove the
Nambu-Poisson version of the corresponding results.

We begin with the following observation which is the Nambu-Poisson version of Lemma 2.2 of \cite{fal-zam}.
Given a canonical subbundle $E \subset TM|_N$ of a Nambu-Poisson manifold $M$ with Nambu tensor $\Pi$, either $\Pi^\sharp (Ann^1 E) \subseteq TN$ or $E = 0$ (cf. Proposition \ref{marsden-ratiu-implication}).
Using the sufficient condition of the Marsden-Ratiu version of Nambu-Poisson reduction theorem \cite{iban-leon-mar-dieg}, we conclude that the 
reduction is always ensured if $E \neq 0$ (cf. Proposition \ref{E= non zero or zero}).

Next we show that the more general Falceto-Zambon Poisson reduction theorem extends naturally for Nambu-Poisson manifolds.
More explicitly, we show that the
canonicity of $E$ and the sufficient condition for the Marsden-Ratiu reduction can be weakend in an appropriate way to ensure the reduction (cf. Theorems \ref{marsden-ratiu-improve}, \ref{mr-improve-thm}).
This refines the reduction of Nambu-Poisson manifolds considered by Ib$\acute{\text{a}}\tilde{\text{n}}$ez et al \cite{iban-leon-mar-dieg}.
The Falceto-Zambon version of reduction theorem involves a smaller subbundle $D \subseteq E \subseteq TM|_N$ such that $E \cap TN \subseteq D$.
We state the Falceto-Zambon version of Nambu-Poisson map reduction and dynamics reduction (cf. Proposition \ref{n-p-map-red}, Theorem \ref{n-p-dynamics-red}), whose proofs are same as the Marsden-Ratiu case.
In the following we also deduce the algebraic interpretation of our reduction theorem (cf. Theorem \ref{algebraic-intrprttn}) and reduction of subordinate
Nambu structures (cf. Proposition \ref{reduction-subordinate}). Motivated from \cite{fal-zam} we give an application of the Falceto-Zambon reduction theorem when the subbundle $D \subset TM|_N$ 
is the restriction of some suitable integrable distribution on $M$ (subsection \ref{subsec-4.5}).

The notion of gauge transformations of Poisson structures associated with certain closed $2$-forms was introduced by {\v S}evera and Weinstein \cite{sev-wein} 
in connection with Poisson-sigma models. Roughly, a gauge transformation modifies a given Poisson structure by adding to its leafwise symplectic structure 
the pullback of the globally defined $2$-form. In this note we introduce gauge transformations of Nambu-Poisson structures (of order $n$) associated with 
certain closed $n$-forms. We show that gauge equivalent Nambu-Poisson structures on a manifold gives rise to same singular foliation, and corresponds to
 isomorphic Leibniz algebroids (cf. Remark \ref{gauge-n-p-same-fol}, Proposition \ref{gauge-n-p-same-leib}). We believe that gauge transformations of Nambu-Poisson structures will have connection with Nambu-sigma models, 
recently considered by B. Jurco and P. Schupp \cite{jurco-schupp1, jurco-schupp2 }. Finally, we show that gauge transformations commute with the reduction procedure (cf. Theorem \ref{gauge-red=red-gauge}).

\noindent{\bf Organization.} In section \ref{sec-2} we recall some basic preliminaries on Nambu-Poisson manifolds and their Marsden-Ratiu reduction. 
In section \ref{sec-3} we show that the Marsden-Ratiu reduction for Nambu-Poisson manifolds is always ensured unless the canonical vector subbundle is zero. Section \ref{sec-4} is devoted to the 
version of Falceto-Zambon  reduction theorem for Nambu-Poisson manifolds. Finally, in section \ref{sec-5} we introduce gauge transformations of 
Nambu-Poisson structures and prove Theorem \ref{gauge-red=red-gauge}.

\section{Nambu-Poisson manifolds and M-R reduction}\label{sec-2}
In this section we recall some basic preliminaries on Nambu-Poisson manifolds \cite{duf-zung, gau, iban-leon-mar-pad} and Marsden-Ratiu reduction \cite{iban-leon-mar-dieg}.

\subsection{Nambu-Poisson manifolds}

\begin{defn}
 Let $M$ be a smooth manifold of dimension $m$. A {\it Nambu-Poisson structure} of order $n$ $(n \leq m)$ on $M$ is an $n$-multilinear skew-symmetric bracket
$$\{, \ldots,\} : C^\infty(M) \times \stackrel{(n)}{\cdots} \times C^\infty(M) \rightarrow C^\infty(M)$$
on the space $C^\infty(M)$ of smooth functions on $M$ satisfying the following:
\begin{itemize}
 \item[(i)] Leibniz rule:
$ \{ f_1, \ldots, f_{n-1}, gh\} = g \{ f_1, \ldots, f_{n-1}, h\} + \{ f_1, \ldots, f_{n-1}, g\} h ,$
 \item[(ii)] fundamental identity (generalization of the Jacobi identity):
$$\{g_1, \ldots, g_{n-1}, \{f_1, \ldots, f_n \} \} = \sum_{k=1}^n \{ f_1, \ldots, f_{k-1}, \{ g_1, \ldots, g_{n-1} , f_k \}, \ldots, f_n \},$$
\end{itemize}
for all $f_1, \ldots, f_n, g_1, \ldots, g_{n-1}, g, h \in C^\infty(M).$ 
\end{defn}

The pair $(M, \{, \ldots,\})$ is called a {\it Nambu-Poisson manifold} of order $n$.
In this paper, by a Nambu-Poisson manifold,
we shall always mean a Nambu-Poisson manifold of order $n$.  See \cite{iban-leon-mar-dieg , iban-leon-mar-pad} for examples of Nambu-Poisson manifolds.
A smooth map between two Nambu-Poisson manifolds of same order $n$ is called a {\it Nambu-Poisson map} if it preserves the corresponding brackets.
A Nambu-Poisson manifold of order $2$ is nothing but a Poisson manifold.

Let $(M, \{,\ldots, \})$ be a Nambu-Poisson manifold of order $n$. Since the bracket is skew-symmetric and satisfies the Leibniz rule, it follows that there exists
a skew-symmetric tensor $\Pi \in \Gamma(\Lambda^nTM)$ of type $(n, 0)$ such that
$$ \Pi (df_1, \ldots, df_n) = \{f_1, \ldots, f_n\},$$
for all $f_1, \ldots, f_n \in C^\infty(M).$ In this case, $\Pi$ is called the corresponding Nambu tensor and the Nambu-Poisson manifold $(M, \{,\ldots,\})$ is also denoted
by $(M, \Pi)$. The tensor $\Pi$ induces a bundle map
 $\Pi^\sharp : \Lambda^{n-1}T^*M \rightarrow TM$ given by
$$\langle \beta, \Pi^\sharp (\alpha_1 \wedge \cdots \wedge \alpha_{n-1}) \rangle = \Pi (\alpha_1, \ldots, \alpha_{n-1}, \beta),~~ \forall \alpha_i, \beta \in T^*M.$$
Given any $(n-1)$ functions $f_1, \ldots, f_{n-1} \in C^\infty(M)$, the {\it Hamiltonian vector
field} associated to these functions, denoted by $X_{f_1 ,\ldots,f_{n-1}}$ and is defined by
$$ X_{f_1 ,\ldots,f_{n-1}} = \Pi^\sharp (df_1 \wedge \cdots \wedge df_{n-1}).$$
Then the fundamental identity in terms of Hamiltonian vector fields can also be rephrased as
\begin{align}\label{fund-hamiltonian}
 [X_{g_1,\ldots,g_{n-1}} , X_{f_1,\ldots,f_{n-1}} ] = \sum_{k=1}^{n-1} X_{f_1,...,\{g_1 , \ldots ,g_{n-1} ,f_k \},...,f_{n-1}}, ~~ \forall g_i, f_j \in C^\infty(M).
\end{align}

The following result describes the local structure of a Nambu-Poisson manifold \cite{duf-zung,gau}.
\begin{thm}\label{nambu-poisson-tensor-decomposable}
 Let $M$ be a smooth manifold of dimension $m$. Then a skew-symmetric $n$-tensor $\Pi \in \Gamma(\Lambda^nTM)$, $n \geq 3$, defines a Nambu-Poisson structure on $M$
if and only if for all $x \in M$ with $\Pi (x) \neq 0$, there exist local coordinates $(U ; x^1, \ldots, x^n, x^{n+1}, \ldots, x^m)$ around $x$
such that
$$ \Pi|_U = \frac{\partial}{\partial x^1} \wedge \cdots \wedge \frac{\partial}{\partial x^n}.$$
\end{thm}

For each $m \in M$, let $\mathcal{D}_m M \subset T_m M$ be the subspace of the tangent space at $m$ generated by all
Hamiltonian vector fields at $m$. It follows from Equation (\ref{fund-hamiltonian})
that $\mathcal{D}(M)$ defines a (singular) integrable distribution, called the {\it characteristic distribution} of $M$,
whose leaves are either $n$-dimensional submanifolds endowed with a volume form or just singletons.

\begin{remark}
It can be shown that, if $\Pi$ is a Nambu tensor of order $\geq 3$ and $g$ is any smooth function, then $g \Pi$ is also a Nambu tensor \cite{duf-zung}.
\end{remark}

Let $M$ be a smooth manifold. Consider the bundle
$$ \mathcal{T}^n (M) = TM \oplus {\Lambda}^{n-1}T^*M.$$
The space of sections of $ \mathcal{T}^n (M)$ carries a higher order Dorfman bracket $\llbracket ~,~ \rrbracket$, given by
\begin{align}\label{dorfman-brk}
\llbracket (X , \alpha) , (Y , \beta) \rrbracket = ([X,Y] , \mathcal{L}_X \beta - i_Y d \alpha),
\end{align}
for $(X , \alpha) , (Y , \beta) \in \Gamma( \mathcal{T}^n (M))$, where $\mathcal{L}$ denotes the Lie derivative and $i$ denotes the contraction operator.

Another characterization of Nambu-Poisson tensor is given by the following \cite{bi-sheng}.
\begin{prop}\label{nambu-characterization-bi-sheng}
 Let $\Pi \in \Gamma(\Lambda^nTM)$ be a skew-symmetric $n$-tensor on $M$, and $\Pi^\sharp : \Lambda^{n-1}T^*M \rightarrow TM$ be the induced bundle map. Then
$$ L_\Pi : = Graph ~ (\Pi^\sharp) = \{ (\Pi^\sharp \alpha, \alpha) |~ \alpha \in {\Lambda}^{n-1}T^*M \} \subset \mathcal{T}^n(M)$$
is closed under the higher order Dorfman bracket $\llbracket ~,~ \rrbracket$ if and only if $\Pi$ is a Nambu-Poisson tensor.
\end{prop}

It follows that, if $(M, \Pi)$ is
a Nambu-Poisson manifold of order $n$, the bundle $\Lambda^{n-1}T^*M \rightarrow M $ carries a Leibniz algebroid structure whose
bracket is given by
\begin{equation}\label{leib-brack}
 \boldsymbol{\{}  \alpha, \beta \boldsymbol{\}}_\Pi  = \mathcal{L}_{\Pi^{\sharp} \alpha} \beta - i_{\Pi^\sharp \beta} d\alpha,
\end{equation}
for all $\alpha , \beta \in \Omega^{n-1}(M)$ and the anchor is given by the map $\Pi^{\sharp}$ \cite{bi-sheng, wade}.

\subsection{Marsden-Ratiu reduction}\label{subsec-m-r}

Let $(M, \{, \ldots,\})$ be a Nambu-Poisson manifold of order $n$ with corresponding Nambu tensor $\Pi$. Let $N \subset M$ be a Nambu-Poisson submanifold and $i : N \hookrightarrow M$ be the inclusion. That is, 
$N$ has a Nambu-Poisson structure such that the inclusion map $i$ is a Nambu-Poisson map.
Therefore, if $h \in C^\infty(M)$ is such that $h|_N \equiv 0$ \big(that is, $dh \in (TN)^0$ \big), then for any
$f_1,\ldots, f_{n-1} \in C^\infty(M),$
$$ \Pi^{\sharp}(df_1 \wedge \cdots \wedge df_{n-1}) (dh)(p) = \{f_1 \circ i,\ldots,f_{n-1} \circ i, h \circ i \}_N(p) = 0, $$
for all $p \in N$, where $\Pi^\sharp : \Lambda^{n-1}T^*M \rightarrow TM$ is the bundle map induced by $\Pi$. Thus, it
implies that $\Pi^{\sharp}(\Lambda^{n-1}T^*_pM) \subseteq T_p N,$ for all $p \in N$. Conversely, if the above relation holds pointwise on a submanifold $N$,
then $N$ induces a Nambu-Poisson structure such that the inclusion map is a Nambu-Poisson map. The induced Nambu structure on $N$ is defined by arbitrary extensions of the functions on $N$.

Next consider a Nambu-Poisson manifold $(M, \{, \ldots, \})$ together with an integrable distribution $E$ which induces a regular foliation $\mathcal{F}$, that is,
the space of leaves $M / \mathcal{F}$ is a smooth manifold and the projecton map $\pi : M \rightarrow M / \mathcal{F}$ is a submersion.
A natural question arises, when $M / \mathcal{F}$ inherits a Nambu-Poisson structure such that $\pi$ is a Nambu-Poisson map. For that, take any $f_1,\ldots,f_{n} \in C^\infty(M / \mathcal{F})$. Then
$f_1 \circ \pi, \ldots, f_n \circ \pi$ are the functions on $M$ which are constant along the fibers of the projection \big(that is, $d(f_k \circ \pi)|_E = 0$\big). In order that $\pi$ is a Nambu-Poisson map, the function $\{f_1 \circ \pi, \ldots, f_n \circ \pi\}$ has to be constant along the fibers,
that is, $(d \{ f_1 \circ \pi, \ldots, f_n \circ \pi \} )|_E = 0.$

Both situations above may be viewed as a particular case of Marsden-Ratiu reduction theorem for Nambu-Poisson manifolds \cite{iban-leon-mar-dieg}.
Let $(M, \Pi)$ be a Nambu-Poisson manifold, $N \subset M$ a submanifold and $i: N \hookrightarrow M$ the inclusion. Let
$E \subset TM|_N$ be a subbundle of $TM$ when restricted to $N$ which satisfies the following condition:

\begin{itemize}
 \item[$\bullet$] $F := E \cap TN $ is a regular, integrable distribution on $N$. Thus, it defines a regular foliation $\mathcal{F}$ on $N$, so
 the space of leaves $\underline{N} := N / \mathcal{F}$ is a smooth manifold with projection map $\pi: N \rightarrow \underline{N}$ is
a submersion.
\end{itemize}

Note that any function on $N$ whose differential vanishes on $F$ can be extended to a function in a neighbourhood $N'$ of $N$ with differential
vanishing on $E$ \cite{fal-zam}. We assume that $N' = M$ has this property. Thus, if $C^\infty(M)_E$ denotes the space of functions on $M$ whose differential
vanish on $E$, the restriction map $i^* : C^\infty(M)_E \rightarrow C^\infty(N)_F$ is surjective.

\begin{defn}
 A triple $(M, N, E)$ with the above properties is called {\it reducible } or {\it Nambu-Poisson reducible} if $\underline{N} = N / \mathcal{F}$ has a Nambu-Poisson structure with the induced bracket $\{, \ldots,\}_{\underline{N}}$ such that
for any $f_1, \ldots, f_n \in C^\infty(\underline{N})$, and any smooth extensions $F_1, \ldots, F_n \in C^\infty(M)_E$ of the functions $f_1 \circ \pi, \ldots, f_n \circ \pi$, respectively,  
we have
$$ \pi^* \{f_1, \ldots, f_n\}_{\underline{N}}  = i^* \{F_1, \ldots, F_n \}  .$$
\end{defn}

\begin{defn}
 A subbundle $E \subset TM|_N$ is called {\it canonical}  if for any smooth functions $F_1, \ldots, F_n $  on $M$ with differentials vanishing on $E$,
 the differential of the function $\{F_1, \ldots, F_n\}$ also vanishes on $E$, that is,
$$ F_1, \ldots, F_n \in C^\infty(M)_E ~ ~ \Rightarrow ~ ~ \{F_1, \ldots, F_n\} \in C^\infty(M)_E .$$
\end{defn}

 The Marsden-Ratiu reduction theorem for Nambu-Poisson manifolds \cite{iban-leon-mar-dieg} is the following.

\begin{thm}\label{marsden-ratiu}
 Let $(M, \{, \ldots,\})$ be a Nambu-Poisson manifold of order $n$ with associated Nambu tensor $\Pi$. 
Let $N \subset M$ be a submanifold, and $E \subseteq TM|_N$ be a canonical subbundle such that
 $F:= E \cap TN$ is a regular, integrable distribution on $N$.
Then the triple $(M, N, E)$ is reducible if and only if
\begin{align*}
 \Pi^{\sharp}(Ann^1 E) \subseteq TN + E ,
\end{align*}
where $Ann^1 E_p = \{ \eta \in \Lambda^{n-1} T_p^*M | ~ i_v \eta = 0 , ~\forall v \in E_p ,~ p \in N\}.$
\end{thm}

Note that, $Ann^1 E_p$ is generated by elements of the form $\alpha_1 \wedge \cdots \wedge \alpha_{n-1}$,
where for all $k=1, \ldots, n-1$; $\alpha_k \in E_p^0=\{ \alpha \in T_p^*M | ~ \alpha (v) = 0, ~\forall v \in E_p,~ p \in N \}$ is the annihilator of $E_p$. When $M$ is a Poisson manifold (that is, when $n=2$), this is the Marsden-Ratiu reduction theorem  for Poisson manifolds \cite{mars-ratiu}.
We remark that the singular version of the Marsden-Ratiu reduction theorem for Nambu-Poisson manifolds has been studied by the author in \cite{das}.

\section{Non-zero $ E \subset TM|_N $}\label{sec-3}

In this section, we show that the triple $(M, N, E)$ as described in subsection \ref{subsec-m-r} is always reducible provided the canonical subbundle $E \neq 0.$

The following is a generalization of Lemma 2.2 of \cite{fal-zam}.

\begin{prop}\label{marsden-ratiu-implication}
 Let $(M, \{, \ldots,\})$ be a Nambu-Poisson manifold with associated Nambu tensor $\Pi$,  and $ N \subset M $ be a submanifold. Assume that $ E \subset TM|_N$ is a canonical subbundle. Then either 
$$ \Pi^{\sharp} (Ann^1 E) \subseteq TN \hspace*{0.2cm} \text{  or } \hspace*{0.2cm} E=0 .$$
\end{prop}

\begin{proof}
Suppose there is a point $p \in N$ such that $\Pi^{\sharp} (Ann^1 E_p) \nsubseteq T_p N$. Since $Ann^1 E_p$ is generated by elements of the form
$\alpha_1 \wedge \cdots \wedge \alpha_{n-1}$, with $\alpha_i \in E_p^0$, for all $i=1, \ldots, n-1$, 
there exist functions $f_1, \ldots, f_{n-1} \in C^\infty(M)$ with differentials vanishing on $E$ such that $\Pi^{\sharp} (df_1 \wedge \cdots \wedge df_{n-1})(p) \notin T_pN$.
Hence, there is a function $g \in C^\infty(M)$ with $g|_N \equiv 0$ \big(that is, $dg(p) \in (T_pN)^0$ \big) such that
$\langle \Pi^{\sharp}(df_1 \wedge \cdots \wedge df_{n-1})(p), dg(p) \rangle \neq 0 ,$
that is,
$ \{f_1,\ldots, f_{n-1},g\}(p) \neq 0 .$
Since $d(g^2) = 2g dg$ and $g$ vanishes on $N$, the differential of the function $g^2$ also vanishes on $E$. As the bundle $E$ is canonical, 
we have $ d\{f_1,\ldots,f_{n-1},g^2\}|_E = 0$. Thus,
$d \big(g\{f_1,\ldots,f_{n-1},g\}\big)|_E = 0 $, which implies that
$$ i_{v}(dg)(p) \{f_1,\ldots,f_{n-1},g\} (p) + i_{v} \big(d\{f_1,\ldots,f_{n-1},g\}\big)(p) g(p) = 0,$$
for all $v \in E_p$. Hence, $i_v (dg)(p) = 0$, for all $v \in E_p$.

Next consider any function $h \in C^\infty(M)$ with $h|_N \equiv 0$. Then the differential of the product function $gh$ vanishes on $E$, as
$ d(gh)|_E = g dh|_E + h dg|_E  \text{~~ and ~~} g|_N = 0 = h|_N .$
Hence, from the canonicity of $E$, the differential of the function $ \{f_1,\ldots,f_{n-1}, gh\}$  also vanishes on $E$, which then
implies that $i_{v}(dh)(p) = 0$, for all $ v \in E_p $. As $(TN)^0$ is locally generated by the differential of functions vanishing on $N$,
it follows that
$ E_p \subseteq T_p N.$
Since the bundle $E \cap TN $ is a smooth distribution of constant rank, we must have $E \subseteq TN $ everywhere. Thus, for  any function
$f \in C^\infty(M)$, the differential of the function $fg$ vanishes on $E$, because
$d(f g)|_E = f dg|_E + g df|_E \text{ and } dg \in (TN)^0 \subseteq E^0 .$
Therefore, $ d \big(\{f_1,\ldots,f_{n-1},f g\}\big)|_E = 0 $. This implies that
$i_{v}(df)(p) = 0 $, for all $v \in E_p$ and $p \in N$. This can happen only when $E_p = 0$, for all $p \in N$, that is, $E = 0$. Hence the proof.
\end{proof}

\begin{exam}\label{n-1, annihilator} 
 Let $(M, \{, \ldots,\})$ be a Nambu-Poisson manifold of order $n$ with induced Nambu tensor $\Pi$, and $N \subset M$ a submanifold. Let
$$ E = \Pi^\sharp (Ann^{n-1}TN),$$
where $ (Ann^{n-1}TN)_p = \{ \eta \in \Lambda^{n-1}T^*_pM | ~ i_{v_1 \wedge \cdots \wedge v_{n-1}} \eta = 0, ~ \forall v_1, \ldots, v_{n-1} \in T_pN ,~ p \in N\}$.
Thus, $E$
is (locally) generated by vector fields $\Pi^\sharp (dh_1 \wedge \cdots \wedge dh_{n-1})$, where $h_1, \ldots, h_{n-1}$ are smooth functions with $dh_i$ is vanishing $TN$, for some $i \in \{1, \ldots, n-1\}.$
The bundle $E$ is canonical as shown in \cite{iban-leon-mar-dieg}.

Moreover, the bundle $E$ satisfies $\Pi^\sharp (Ann^1 E) \subset TN$.  
One can also conclude the same fact by using Proposition \ref{marsden-ratiu-implication}.
\end{exam}

\begin{remark}Let $(M, \{, \ldots,\})$ be a Nambu-Poisson manifold of order $n$ with associated tensor $\Pi$.
 Consider the Leibniz bracket ${\textbf \{} ~, {\textbf \}}_\Pi$ on the space of $(n-1)$ forms on $M$ given by Equation (\ref{leib-brack}).
This bracket satisfies
\begin{equation}\label{leib-brack-explct}
{\textbf \{} df_1 \wedge \cdots \wedge df_{n-1}, dg_1 \wedge \cdots \wedge dg_{n-1} {\textbf \}}_\Pi = \sum_{i=1}^{n-1} dg_1 \wedge \cdots \wedge d\{f_1,\ldots,f_{n-1},g_i\} \wedge \cdots \wedge dg_{n-1},
\end{equation}
for all $f_1, \ldots, f_{n-1}, g_1, \ldots, g_{n-1} \in C^\infty(M).$ Let $N \subset M$ be a submanifold, and $0 \neq E \subseteq TM|_N$ be a canonical subbundle.
Since $Ann^1 E$ is generated by elements of the form $df_1 \wedge \cdots \wedge df_{n-1}$,
where $f_1, \ldots, f_{n-1}$ are smooth functions on $M$ with differentials $df_k$ vanish on $E$, it follows from  Equation (\ref{leib-brack-explct}) and the canonicity of $E$ that the sections of the subbundle
$(Ann^1 E) \rightarrow N$ are closed with respect to the bracket defined by (\ref{leib-brack}). Moreover, from Proposition \ref{marsden-ratiu-implication}, the anchor
$\Pi^\sharp$ maps $(Ann^1 E$) to $TN$. Hence, $(Ann^1E) \rightarrow N$ is a Leibniz subalgebroid of $\Lambda^{n-1}T^*M \rightarrow M $.
\[
\xymatrixrowsep{0.5in}
\xymatrixcolsep{0.7in}
\xymatrix{
Ann^1 E \ar@{^{(}->}[r] \ar[d] & \Lambda^{n-1}T^*M \ar[d] \\
N  \ar@{^{(}->}[r] & M
}
\]
We have a different Leibniz algebroid structure on $\Lambda^{n-1}T^*M \rightarrow M $ associated to any Nambu-Poisson manifold of order $n$, given by
Ib$\acute{\text{a}}\tilde{\text{n}}$ez et al \cite{iban-leon-mar-pad}. The Leibniz bracket as defined in \cite{iban-leon-mar-pad} also satisfies Equation (\ref{leib-brack-explct}). Thus, in this case, the bundle $(Ann^1E) \rightarrow N$ is a Leibniz subalgebroid of  $\Lambda^{n-1}T^*M \rightarrow M $.
\end{remark}

Combinding Theorem \ref{marsden-ratiu} and Proposition \ref{marsden-ratiu-implication}, we get the following result which is analogous to Theorem 2.2 of
\cite{fal-zam}.
\begin{prop}\label{E= non zero or zero}
Let $E \subseteq TM|_N $ be a canonical subbundle such that $F:= E \cap TN$ is a regular, integrable distribution on $N$.
\begin{enumerate}
 \item If $E \neq 0$, then $(M, N, E)$ is reducible.
 \item If $E = 0$, then $(M, N, E)$ is reducible if and only if $\Pi^{\sharp}(\Lambda^{n-1} T_p^{*}M) \subseteq T_p N $, for all $p \in N$, that is, if and only if $N$ is a Nambu-Poisson submanifold.
\end{enumerate}
\end{prop}

\begin{remark}
It follows from Proposition \ref{E= non zero or zero} that the triple $(M, N, E = 0)$ is reducible if and only if $N$ is a Nambu-Poisson submanifold. If $E'$ is any canonical subbundle such that
$E' \cap TN = 0$, the Nambu-Poisson structures on $N$ induced by $E$ and $E'$ are the same, as $N$ is a Nambu-Poisson submanifold.

\end{remark}

A {\it Nambu ring} is an associative, commutative ring $\mathcal{R}$ with a skew-symmetric $n$-multilinear bracket
\begin{align*}
 \{,\ldots,\} : \underbrace{ \mathcal{R} \times \cdots \times \mathcal{R}}_{n \text{~copies}} \rightarrow \mathcal{R}
\end{align*}
which satisfies the Leibniz rule and the fundamental identity. A subring of $\mathcal{R}$ is called a {\it Nambu subring} if it is itself a Nambu ring under the induced structure.

Let $\mathcal{R}$ be a Nambu ring and $\mathcal{I}$ be an ideal of it. Given a Nambu 
subring $\mathcal{N}$, the quotient $\mathcal{N} / (\mathcal{N} \cap \mathcal{I})$ inherits a Nambu ring structure \cite{mar-ibo}.

\begin{remark}		
Let $E \subset TM|_N$ be a canonical subbundle. Then the induced Nambu-Poisson structure on $C^\infty(N / \mathcal{F}) = C^\infty(M)_E / (C^\infty(M)_E  \cap \mathcal{I})$ given by Proposition
\ref{E= non zero or zero}, is just the quotient Nambu structure as above, where $\mathcal{I}$ is the ideal of smooth functions on $M$ vanishing on $N$.
\end{remark}

\section{Falceto-Zambon reduction}\label{sec-4}
In this section, we study the version of Falceto-Zambon Poisson reduction theorem for Nambu-Poisson manifolds, and subsequently we deduce the algebraic interpretation of our
main result  and reduction of subordinate Nambu structures. Our approaches here closely follow the work of Falceto and Zambon for Poisson manifolds \cite{fal-zam}.

\subsection{Falceto-Zambon reduction}

In the Marsden-Ratiu reduction theorem for Nambu-Poisson manifolds, the induced bracket on $C^\infty(\underline{N})$ is given as follows.
For any $f_1, \ldots, f_{n} \in C^\infty(\underline{N})$, choose any arbitrary extensions $F_1, \ldots, F_n \in C^\infty(M)_E$ of
$f_1 \circ \pi, \ldots, f_n \circ \pi$. The bracket $\{, \ldots,\}_{\underline{N}}$ on $C^\infty(\underline{N})$ is then defined by
\begin{align}\label{defn-m-r-bracket}
 \{f_1, \ldots, f_n \}_{\underline{N}} := i^* \{F_1, \ldots, F_n \}.
\end{align}
Note that the function on the right hand side of the above expression is in $C^\infty(N)_F$. To prove that the above bracket is well defined, one
uses the fact that given any two extensions $F_n$ and $F_n'$ of $f_n \circ \pi$, the differential $d (F_n - F_n')$ annihilate both $TN$ and $E$, 
thus, annihilate $TN + E$. On the other hand, $$\Pi^\sharp (dF_1 \wedge \cdots \wedge dF_{n-1}) \in \Pi^\sharp (Ann^1 E) \subset TN + E.$$
Thus, it follows that the bracket is independent of the chosen extensions. This independence is even valid if there is a subbundle
$D \subset TM|_N$ such that $F \subseteq D \subseteq E$ and satisfying $\Pi^\sharp(Ann^1E) \subset TN + D$. 
To verify the fundamental identity of the reduced bracket, one observes that the canonicity
of $E$ may be weakend. More precisely, we only need the fact that if $F_1, \ldots, F_n \in C^\infty(M)_E$, their bracket $\{F_1, \ldots, F_n\}$ is
in $C^\infty(M)_D$. One can also improve the reduction by adding a multiplicative subalgebra $\mathcal{B} \subset C^\infty(M)_E$
such that the restriction map $i^* : \mathcal{B} \rightarrow C^\infty(N)_F$ is surjective.
 
With the above observations, we have the following Nambu-Poisson version of Falceto-Zambon reduction theorem (compare with Theorem 3.1 \cite{fal-zam}).
\begin{thm}\label{marsden-ratiu-improve}
 Let $(M, \{, \ldots,\})$ be a Nambu-Poisson manifold with associated Nambu tensor $\Pi$, and $N \subset M$ be a submanifold. Let $E \subset TM|_N $ be a subbundle
(may not be canonical) such that $F := E \cap TN $  is a regular, integrable distribution. Assume that $D$ is a subbundle of $TM|_N$ satisfying 
$F \subseteq D \subseteq E $ and
\begin{equation}\label{marsden-ratiu-improve-1}
 \Pi^{\sharp} (\text{Ann}^1 E) \subseteq TN + D .
\end{equation}
Moreover, let $\mathcal{B} \subseteq C^\infty(M)_E$ be a multiplicative subalgebra such that the restriction map $i^* : \mathcal{B} \rightarrow C^\infty(N)_F$
is surjective and
\begin{equation}\label{marsden-ratiu-improve-2}
 \{\mathcal{B}, \ldots, \mathcal{B}\} \subseteq C^\infty(M)_D
\end{equation}
holds. Then $(M, N, E)$ is reducible.
\end{thm}

\begin{proof}
 Let $f_1, \ldots, f_n \in C^\infty(\underline{N}) = C^\infty(N)_F $ be any functions on $\underline{N}$ and choose their arbitrary extensions $F_1, \ldots, F_n $ in $\mathcal{B}$.
Then the bracket $\{, \ldots,\}_{\underline{N}}$ on $C^\infty(\underline{N})$ is defined by Equation (\ref{defn-m-r-bracket}).
Suppose there is another extension $F_n^\prime \in \mathcal{B} \subseteq C^\infty(M)_E$
 for $f_n$. 
Then the differential of the function
$(F_n - F_n^\prime)$ annihilates $TN + E $.

On the other hand, since each $F_k \in C^\infty(M)_E $, we have $dF_1 \wedge \cdots \wedge dF_{n-1} \in Ann^1 E .$ Therefore,
$$ \Pi^{\sharp}(dF_1 \wedge \cdots \wedge dF_{n-1}) \in \Pi^{\sharp} (Ann^1 E)  \subseteq TN + D  \subseteq TN + E,$$
which implies that
$ i^*\{F_1,\ldots,F_{n-1},F_n - {F_n}^\prime \} = i^* \langle \Pi^{\sharp} (dF_1 \wedge \cdots \wedge dF_{n-1}), d(F_n - {F_n}^\prime)\rangle = 0.$
Thus, by skew-symmetry, the bracket is independent of the chosen extensions of its entries. Hence, the bracket $ \{f_1, \ldots, f_n\}_{\underline{N}} $ is well defined. The property of skew-symmetryness and the Leibniz rule of this bracket follows from that of $\{, \ldots,\}$.

To prove the fundamental identity of this bracket, we need the following observation. Let $\{f_1, \ldots, f_n\}^{\mathcal{B}}_{\underline{N}}$ be any
 extension of $\{f_1, \ldots, f_n\}_{\underline{N}}$ in $\mathcal{B}$. Then from the definition of the bracket $\{f_1, \ldots, f_n\}_{\underline{N}}$, 
it follows that the functions $\{F_1, \ldots, F_n\}$ and $\{f_1, \ldots, f_n\}^{\mathcal{B}}_{\underline{N}}$ agrees on $N$. Thus,
$d \big(\{F_1, \ldots, F_n\}   -  \{f_1, \ldots, f_n\}^{\mathcal{B}}_{\underline{N}}\big) \in (TN)^0.$
Moreover, since the function $\{F_1, \ldots, F_n\}$ is in $C^\infty(M)_D$, and the function $\{f_1, \ldots, f_n\}^{\mathcal{B}}_{\underline{N}}$
is in $\mathcal{B} \subseteq C^\infty(M)_E$, we have
$ d\big(\{F_1, \ldots, F_n\}   -  \{f_1, \ldots, f_n\}^{\mathcal{B}}_{\underline{N}}\big)|_D = 0 .$ Thus, the differential of the function $\big( \{F_1, \ldots, F_n\}   -  \{f_1, \ldots, f_n\}^{\mathcal{B}}_{\underline{N}} \big)$  annihilates both $TN$ and $D$, hence, annihilates
$TN + D$. Thus, it follows from condition (\ref{marsden-ratiu-improve-1}) that the bracket of $(n-1)$ functions of $C^\infty(M)_E$ with the above difference function is zero.
Therefore, for any $g_1, \ldots, g_{n-1} \in C^\infty(\underline{N}),$
\begin{align*}
\{g_1,\ldots,g_{n-1},\{f_1, \ldots, f_n\}_{\underline{N}}\}_{\underline{N}} =&~ i^*\{G_1,\ldots,G_{n-1},\{f_1, \ldots, f_n\}^{\mathcal{B}}_{\underline{N}}\}\\
=&~ i^*\{G_1, \ldots ,G_{n-1},\{F_1, \ldots, F_n\}\} .
\end{align*}
Thus, the fundamental identity of the bracket $\{, \ldots,\}_{\underline{N}}$ follows from that of $\{, \ldots,\}.$
\end{proof}

By choosing smaller $D$, we get better improvement of the reduction problem. Taking $D = F$ and $\mathcal{B} = C^\infty(M)_E$, we get a slight
improvement of the Marsden-Ratiu reduction theorem (cf. Theorem \ref{marsden-ratiu}) for Nambu-Poisson manifolds.

\begin{thm}\label{mr-improve-thm}
 Let $(M, \{, \ldots,\})$ be a Nambu-Poisson manifold with associated Nambu tensor $\Pi$, and $N \subset M$ a submanifold. Let  $E \subset TM|_N $ be a subbundle
(may not be canonical) such that 
 $F := E \cap TN $ is a regular, integrable distribution on $N$ and that
\begin{itemize}
 \item[(i)] if $F_1, \ldots, F_n \in C^\infty(M)_E$ are smooth functions on $M$, then
$$\{F_1 ,\ldots, F_n \} \in C^\infty(M)_F .$$
 \item[(ii)] Moreover,
\begin{align}\label{ann-subset-tn}
 \Pi^{\sharp} (Ann^1 E) \subseteq TN.
\end{align}
\end{itemize}
Then $(M, N, E)$ is reducible.
\end{thm}

\begin{remark}
In the above theorem, condition $(ii)$ is equivalent to the following: locally there exists a frame of sections $X_i$ of $F$ and for any extensions of them to vector fields on $M$ such that
$$ (\mathcal{L}_{X_i} \Pi)|_N \subseteq E \wedge {\bigwedge}^{n-1}TM|_N .$$
This follows from the formula of the Lie derivative
\begin{align*}
 (\mathcal{L}_{X_i} \Pi)(dF_1, \ldots, dF_n) &= X_i (\Pi (dF_1, \ldots, dF_n)) - \sum_{k=1}^n \Pi (dF_1, \ldots, d (X_i(F_k)), \ldots, dF_n) \\
&= X_i (\{F_1, \ldots, F_n\}) - \sum_{k=1}^n \{    F_1, \ldots, X_i(F_k), \ldots, F_n \}.
\end{align*}
Indeed, if $F_1, \ldots, F_n \in C^\infty(M)_E$ and $\{F_1, \ldots, F_n\} \in C^\infty(M)_F$, the first term of the right hand side vanishes. Moreover, in the right hand side, each term of the summation vanishes on $N$ as $X_i (F_k)|_N = \langle X_i , dF_k \rangle = 0$ and $\Pi^\sharp (Ann^1E) \subset TN$. Therefore, we have
$$ (\mathcal{L}_{X_i} \Pi)(dF_1, \ldots, dF_n)|_N = 0 ,$$
which implies that
$$ (\mathcal{L}_{X_i} \Pi)|_N \subseteq E \wedge {\bigwedge}^{n-1}TM|_N .$$
\end{remark}

\begin{corollary}
 Let $(M, N, E)$ be a triple satisfying conditions of Theorem \ref{mr-improve-thm}, so that it is Nambu-Poisson reducible to $(\underline{N}, \Pi_{\underline{N}})$. Then
$$ \Pi (\widetilde{\pi^*\alpha_1}, \ldots, \widetilde{\pi^*\alpha_{n}} ) \circ i = \Pi_{\underline{N}} (\alpha_1, \ldots, \alpha_n) \circ \pi$$
and
$$ d\pi \circ \Pi^\sharp (\widetilde{\pi^*\alpha_1} \wedge \cdots \wedge \widetilde{\pi^*\alpha_{n-1}}) = (\Pi_{\underline{N}})^\sharp (\alpha_1 \wedge \cdots \wedge \alpha_{n-1}) \circ \pi,$$
for any $\alpha_1, \ldots, \alpha_n \in \Omega^1(\underline{N})$ and any extensions $\widetilde{\pi^*\alpha_1}, \ldots, \widetilde{\pi^*\alpha_{n}}
\in \Omega^1(M)$ of $\pi^*\alpha_1, \ldots, \pi^*\alpha_{n}$ vanishing on $E$.
\end{corollary}

Motivated from the examples of \cite{fal-zam} here we provide examples where the assumptions of Theorem \ref{mr-improve-thm} are satisfied but $E$ is not canonical.
\begin{exam}
 Consider the manifold $M = \mathbb{R}^4$ with the Nambu
structure $\Pi = w \frac{\partial}{\partial x} \wedge \frac{\partial}{\partial y} \wedge \frac{\partial}{\partial z}$  of order $3$. Let $N = \{ w = 0 \}$ be the submanifold of $xyz$-plane and $E = \mathbb{R} \frac{\partial}{\partial w}$. The bundle $E$ is not
canonical. Indeed, the functions $f_1 = x$, $f_2 = y$ and $f_3 = z$ which are in $C^\infty(M)_E$, while their Nambu bracket
$$ \{ x, y, z \} =  \bigg( i_{dx \wedge dy \wedge dz} w \frac{\partial}{\partial x} \wedge \frac{\partial}{\partial y} \wedge \frac{\partial}{\partial z}\bigg) = w ,$$
is not in $C^\infty(M)_E$, but in $C^\infty(M)_F$, since $F = E \cap TN = \{ 0 \}.$ The bundle $E$ also
satisfies condition (\ref{ann-subset-tn}) as $\Pi$ vanishes at points of $N$.
\end{exam}

One can extend the preceding example to the case of Nambu structure of higher order.
\begin{exam}
 Take $M = \mathbb{R}^{n+k}$ with the Nambu structure $\Pi = x_{n+1} \frac{\partial}{\partial x_1} \wedge \cdots \wedge \frac{\partial}{\partial x_n}$
 of order $n$. Take $N = \{ x_{n+1}= 0 \}$ and $E = \mathbb{R} \frac{\partial}{\partial x_{n+1}}$.
The bundle $E$ is not canonical as in the previous example and also satisfies condition (\ref{ann-subset-tn}) as $\Pi$ vanishes at points of $N$.

If we take $N' = \{ x_{n+1} = \cdots = x_{n+k} = 0 \},$ the hypothesis of Theorem \ref{mr-improve-thm} also holds. Therefore, the Nambu structure can
also be reduced.
\end{exam}

Similar to the Marsden-Ratiu version of Nambu-Poisson map reduction and dynamics reduction \cite{iban-leon-mar-dieg}, one can state the Falceto-Zambon version as follows.
The proof of these results are same as the Marsden-Ratiu case.

\begin{prop}\label{n-p-map-red}
(reduction of Nambu-Poisson map) Let the tuples $(M_j, N_j, E_j, D_j, \mathcal{B}_j)$ satisfies the conditions of Theorem \ref{marsden-ratiu-improve}, thus $(M_j, N_j, E_j)$ are
Nambu-Poisson reducible, for $j=1,2$. Let $\phi : M_1 \rightarrow M_2$ be a Nambu-Poisson map such that $\phi (N_1) \subseteq N_2$, $\phi_* (E_1) \subset E_2$,
and $\phi^* \mathcal{B}_2 \subset \mathcal{B}_1$. Then $\phi$ induces a unique Nambu-Poisson map $\hat{\phi} : \underline{N_1} \rightarrow \underline{N_2}$
such that $\pi_2 \circ \phi|_{N_1} = \hat{\phi} \circ \pi_1 .$
\end{prop}

Let $M$ be a Nambu-Poisson manifold of order $n$. Then a submanifold $N \subset M$ is {\it conserved} for the functions $F_1, \ldots, F_{n-1} \in C^\infty(M)$,
if $X_{F_1, \ldots, F_{n-1}} (x) \in T_xN$, for all $x \in N$.

\begin{thm}\label{n-p-dynamics-red}
(reduction of dynamics) Let the tuple $(M, N, E, D, \mathcal{B})$ satisfies the conditions of Theorem \ref{marsden-ratiu-improve}, thus $(M, N, E)$ is reducible. Let $H_1, \ldots, H_{n-1} \in \mathcal{B} \subset C^\infty(M)_E$
be a family of functions for which the submanifold $N$ is conserved. In addition, assume that the flow $\phi_t$ of $X_{H_1, \ldots, H_{n-1}}$
preserves the subbundle $E$ and that $\phi_t^* \mathcal{B} \subset \mathcal{B}$. Then $\phi_t$ induces Nambu-Poisson diffeomorphisms $\hat{\phi_t}$ on $\underline{N}$
and $\hat{\phi_t}$ is the flow of the Hamiltonian vector field $X_{h_1, \ldots, h_{n-1}}$, where $h_i \in C^\infty(\underline{N})$ are uniquely determined by
$h_i \circ \pi = H_i |_N$, for $i = 1, \ldots, n-1.$ Moreover, the vector fields $(X_{H_1, \ldots, H_{n-1}}) |_N$ and $X_{h_1, \ldots, h_{n-1}}$ are $\pi$-related.
\end{thm}

\subsection{Algebraic interpretation of Theorem \ref{marsden-ratiu-improve}}

An algebraic formulation of the Marsden-Ratiu reduction theorem for Nambu-Poisson manifolds (cf. Theorem \ref{marsden-ratiu}) is given as follows.
Let $\mathcal{I}$ be an ideal of a Nambu ring $\mathcal{R}$ and $\mathcal{B} \subset \mathcal{R}$ be a Nambu subring such that
\begin{align*}
 \{\mathcal{B},\ldots,\mathcal{B}, \mathcal{B} \cap \mathcal{I}\} \subset \mathcal{I}.
\end{align*}
Then there is an induced Nambu ring structure on 
$\mathcal{B}/({\mathcal{B} \cap \mathcal{I}}).$

In the next, we will give the algebraic interpretation of Theorem \ref{marsden-ratiu-improve} which is similar to Proposition A.1 of \cite{fal-zam}.
\begin{thm}\label{algebraic-intrprttn}
 Let $\mathcal{R}$ be a Nambu ring and $\mathcal{I}$ be an ideal of $\mathcal{R}$. Let $\mathcal{B} \subset \mathcal{D}$ be multiplicative subalgebras
of $\mathcal{R}$ having the same images under the projection map $\mathcal{R} \rightarrow   \mathcal{R}/ \mathcal{I}$. Moreover, suppose that
\begin{equation}\label{marsden-ratiu-improve-alg-1}
 \{\mathcal{B},\ldots,\mathcal{B}, \mathcal{D} \cap \mathcal{I}\} \subset \mathcal{I}
\end{equation}
and
\begin{equation}\label{marsden-ratiu-improve-alg-2}
  \{\mathcal{B}, \ldots, \mathcal{B}\} \subset \mathcal{D} .
\end{equation}
Then $\mathcal{B}/(\mathcal{B} \cap \mathcal{I})$ inherits a Nambu ring structure with its bracket is determined by the following diagram
\[
\xymatrixrowsep{0.5in}
\xymatrixcolsep{0.7in}
\xymatrix{
{\mathcal{B} \times \cdots \times \mathcal{B}} \ar[r]^{\{, \ldots ,\}} \ar[d] & \mathcal{D} \ar[d] \\
\frac{\mathcal{B}}{\mathcal{B} \cap \mathcal{I}} \times \cdots \times  \frac{\mathcal{B}}{\mathcal{B} \cap \mathcal{I}} \ar[r] & \frac{\mathcal{B}}{\mathcal{B} \cap \mathcal{I}} = \frac{\mathcal{D}}{\mathcal{D} \cap \mathcal{I}}.
}
\]
\end{thm}

\begin{proof}
 The bracket on $\mathcal{B}/(\mathcal{B} \cap \mathcal{I}) $ is well defined because of conditions (\ref{marsden-ratiu-improve-alg-1}) and 
(\ref{marsden-ratiu-improve-alg-2}). The induced bracket on $\mathcal{B}/(\mathcal{B} \cap \mathcal{I})$ is of course skew-symmetric and satisfies
the Leibniz rule as so does the bracket on $\mathcal{R}$. To check the fundamental identity of the bracket, consider any
$g_1, \ldots, g_{n-1}, f_1, \ldots, f_n \in \mathcal{B}/(\mathcal{B} \cap \mathcal{I})$ and arbitrary representations
$\widetilde{g_1}, \ldots, \widetilde{g_{n-1}}, \widetilde{f_1}, \ldots, \widetilde{f_n} \in \mathcal{B}$ of them. Then 
$\widetilde{ \{f_1, \ldots, f_n\}  }$ and $\{ \widetilde{f_1}, \ldots, \widetilde{f_n}  \}$ represents the same element, therefore, their difference
$ \widetilde{ \{f_1, \ldots, f_n\}  }  -   \{ \widetilde{f_1}, \ldots, \widetilde{f_n}  \}  $ lies in $\mathcal{B} \cap \mathcal{I}   \subseteq \mathcal{D} \cap \mathcal{I}$.

Thus, we have
\begin{align*}
 \{g_1, \ldots, g_{n-1} , \{f_1, \ldots, f_n \} \} =&~ \{ \widetilde{g_1}, \ldots, \widetilde{g_{n-1}}, \widetilde {\{f_1, \ldots, f_n \}}\}   \quad (\text{mod ~} \mathcal{D} \cap \mathcal{I}) \\
=&~ \{ \widetilde{g_1}, \ldots, \widetilde{g_{n-1}}, \{ \widetilde{f_1}, \ldots, \widetilde{f_n}  \} \}  \quad (\text{mod ~} \mathcal{D} \cap \mathcal{I}).
\end{align*}
Hence, the bracket on $\mathcal{B}/(\mathcal{B} \cap \mathcal{I})$ satisfies the fundamental identity.
\end{proof}

\begin{remark}
 Theorem \ref{marsden-ratiu-improve} can be recovered from the above algebraic formulation by taking $\mathcal{R} = C^\infty(M)$, 
$\mathcal{I} = \{ f \in C^\infty(M)|~ f|_N \equiv 0 \},$ $\mathcal{B} = \mathcal{B}$ and $\mathcal{D} = C^\infty(M)_D$. In this case, conditions 
(\ref{marsden-ratiu-improve-1}) and (\ref{marsden-ratiu-improve-2}) become conditions (\ref{marsden-ratiu-improve-alg-1}) and
(\ref{marsden-ratiu-improve-alg-2}), respectively.
\end{remark}

\subsection{Reduction of subordinate Nambu structures}

Let the tuple $(M, N, E, D, \mathcal{B})$ satisfies the conditions of
Theorem \ref{marsden-ratiu-improve}, that is, there is a subbundle $D \subset TM|_N$ satisfying $F \subseteq D \subseteq E $ and 
$$ \Pi^{\sharp} (Ann^1 E) \subseteq TN + D .$$
Moreover, $\mathcal{B} \subseteq C^\infty(M)_E$ is a multiplicative subalgebra such that the map $i^* : \mathcal{B} \rightarrow C^\infty(N)_F$
is surjective and
$ \{\mathcal{B}, \ldots, \mathcal{B}\} \subseteq C^\infty(M)_D . $
Let $F_1,\ldots,F_{k}$ $(k \leqslant n-2)$ be any fixed functions on $M$. Then there is an induced
Nambu-Poisson structure of order $n-k$ on $M$, called the {\it subordinate} Nambu structure with subordinate functions $F_1,\ldots,F_{k}$ and is defined by
$$ \{f_1,\ldots, f_{n-k}\}_{F_1 \cdots F_{k}} = \{F_1,\ldots,F_{k},f_1,\ldots, f_{n-k}\} ,$$
for any $f_1, \ldots, f_{n-k} \in C^\infty(M)$. The Nambu tensor $\Pi_{F_1 \cdots F_{k}}$ of this subordinate Nambu structure is given by
$$(\Pi_{F_1 \cdots F_{k}})^{\sharp} (df_1 \wedge \cdots \wedge df_{n-k-1}) = \Pi^{\sharp} (dF_1 \wedge \cdots \wedge dF_k \wedge df_1 \wedge \cdots \wedge df_{n-k-1}).$$ Observe that
$$ (\Pi_{F_1 \cdots F_{k}})^{\sharp} (Ann^1_k E) \subset \Pi^{\sharp} (\text{Ann}^1 E) \subset TN + D ,$$
where $(Ann^1_k E)_p = \{ \mu \in \Lambda^{n-k-1} T_x^*M |~ i_v \mu = 0 , \forall v \in E_p \},$ for $p \in N,$
if the differentials $dF_1, \ldots, dF_k$ are vanishing on $E$.
Thus, condition (\ref{marsden-ratiu-improve-1}) holds for this subordinate Nambu structure if $F_1, \ldots, F_k \in C^\infty(M)_E$ . If $F_1,\ldots, F_k$ are also in $\mathcal{B}$, that is, they lie in the same multiplicative subalgebra, then condition (\ref{marsden-ratiu-improve-2})
also holds for the subordinate Nambu structure. In that case, the subordinate Nambu structure $\{, \ldots,\}_{F_1 \cdots F_{k}}$ is also reducible.

Thus we have the following result.
\begin{prop}\label{reduction-subordinate}
 Let the assumptions of Theorem \ref{marsden-ratiu-improve} hold and let $F_1, \ldots, F_k$ be any fixed functions on $M$. If $F_1, \ldots, F_k$ are in  $\mathcal{B}$, then
the subordinate Nambu structure $\{, \ldots,\}_{F_1 \cdots F_{k}}$ is reducible.
\end{prop}

\subsection{Application}\label{subsec-4.5}

In this subsection, we give an application of Theorem \ref{marsden-ratiu-improve} where the subbundle $D \subset TM|_N$ is the restriction of some suitable integrable
distribution on $M$. We recall the following definition from \cite{fal-zam}.

\begin{defn}
 Let $M$ be a manifold, $N \subset M$ a submanifold and $i: N \hookrightarrow M$ the inclusion. Let $E \subseteq TM|_N$ be a subbundle such that 
$F : = E \cap TN$ is a regular, integrable distribution on $N$. If $\theta_D$ is an integrable distribution on $M$ with $F \subseteq D : = \theta_D |_N \subseteq E$, then $\theta_D$ is said to be {\it compatible} with $E$ if the restriction map
$$ i^* : C^\infty(M)_E \cap C^\infty(M)_{\theta_D}  \rightarrow C^\infty(N)_F $$
is surjective.
\end{defn}

In order to provide some examples we need the following proposition which is a generalization of Proposition 4.2 of \cite{fal-zam}.

\begin{prop}\label{last-prop}
 Let $(M, \{,\ldots,\})$ be a Nambu-Poisson manifold with associated Nambu tensor $\Pi$, $N \subset M$ a submanifold, and $i : N \hookrightarrow M$ the inclusion.
Let $E , ~\theta_D $ as in the above
and $\theta_D$ is compatible with $E$. If
\begin{align}\label{ann-subset-tn+d}
\Pi^{\sharp} (Ann^1 E) \subseteq TN + D
\end{align}
and for any section $X \in \Gamma(\theta_D)$,
\begin{align}\label{last-eqn}
 (\mathcal{L}_X \Pi)|_N \subseteq E \wedge {\bigwedge}^{n-1}TM|_N ,
\end{align}
then $(M, N, E)$ is reducible.
\end{prop}

\begin{proof}
 Take $\mathcal{B} = C^\infty(M)_E \cap C^\infty(M)_{\theta_D}$ in Theorem \ref{marsden-ratiu-improve} . As $\theta_D$ is compatible with $E$, the map
$i^* : \mathcal{B} \rightarrow C^\infty(N)_F$ is surjective. Moreover, the condition
(\ref{marsden-ratiu-improve-1})
is given. To show that $\mathcal{B}$ satisfies condition (\ref{marsden-ratiu-improve-2}), that is,
$\{ \mathcal{B}, \ldots, \mathcal{B} \} \subseteq C^\infty(M)_D ,$
take any $F_1, \ldots, F_n \in \mathcal{B} = C^\infty(M)_E \cap C^\infty(M)_{\theta_D}$ and $X \in \Gamma(\theta_D)$.

We have the formula for the Lie derivative
\begin{align*}
 (\mathcal{L}_X \Pi)(dF_1, \ldots, dF_n) = X (\{F_1, \ldots, F_n\}) - \sum_{k=1}^n \{ F_1, \ldots, X(F_k), \ldots, F_n \}.
\end{align*}
Note that, if $(\mathcal{L}_X \Pi)|_N \subset E \wedge {\bigwedge}^{n-1}TM|_N$, then the left hand side of the above expression vanishes on $N$. Moreover, in the right hand side, each term of the summation vanishes, as
$X(F_k) = \langle X , dF_k \rangle = 0$ \big(since $X \in \Gamma (\theta_D)$ and $F_k \in C^\infty(M)_{\theta_D}$\big).
Thus, we have
$(X \{ F_1, \ldots, F_n \})|_N \equiv 0,$
that is $\langle X , d \{F_1, \ldots, F_n \} \rangle |_N = 0$. Since $X|_N \in D = \theta_D |_N$, the differential $d \{F_1, \ldots, F_n \}$ vanishes on $D$. That is,
$\{F_1, \ldots, F_n \} \in C^\infty(M)_D$. Hence, by Theorem \ref{marsden-ratiu-improve}, the triple $(M, N, E)$ is reducible. 
\end{proof}

\begin{exam}
 Consider $M = \mathbb{R}^4$ with the Nambu structure $\Pi = \frac{\partial}{\partial x} \wedge \frac{\partial}{\partial y} \wedge \frac{\partial}{\partial z}$ of order $3$. Take
the submanifold $N = \{ z = 0 \}$ and $E = D = \theta_D|_N$, where $\theta_D = \mathbb{R} \frac{\partial}{\partial z}.$ Then $\theta_D$ is compatible with $E$, and since
$TN + D = TM|_N$,  condition (\ref{ann-subset-tn+d}) also holds. Moreover, we have
$\mathcal{L}_{\frac{\partial}{\partial z}} \Pi = 0.$ Therefore, by Proposition \ref{last-prop}, the triple $(M, N, E)$ is reducible.

Note that in this example the condition $\Pi^\sharp(Ann^1E) \subset TN$ is not satisfied, since $\Pi^\sharp(dx \wedge dy) = \frac{\partial}{\partial z} \notin TN .$
\end{exam}

The next example shows that conditions in Proposition \ref{last-prop} are not necessary in order to obtain a reduced Nambu structure.

\begin{exam}
 Consider $M = \mathbb{R}^4$ with the Nambu structure $\Pi = w \frac{\partial}{\partial x} \wedge \frac{\partial}{\partial y} \wedge \frac{\partial}{\partial z} $ of order $3$. Take the submanifold
$N = \{ w = x \}$ and $E = \mathbb{R} \frac{\partial}{\partial w}$. Then a distribution $\theta_D$ as in Proposition \ref{last-prop} does not exist.
Indeed, if $\theta_D$ exists, it has to be one dimensional because of $E \cap TN \subset \theta_D|_N \subset E$ and condition (\ref{ann-subset-tn+d}). For any vector
field $X$ which restricts to $\frac{\partial}{\partial w}$ on $N$, we have 
$(\mathcal{L}_X \Pi)(p) = X(w)(p) \frac{\partial}{\partial x} \wedge \frac{\partial}{\partial y} \wedge \frac{\partial}{\partial z}  = \frac{\partial}{\partial x} \wedge \frac{\partial}{\partial y} \wedge \frac{\partial}{\partial z}$
at the point $p \in N$, but this is not in $E \wedge \bigwedge^{n-1} TM|_N$. Thus condition (\ref{last-eqn}) is not satisfied. However Equation
(\ref{defn-m-r-bracket}) defines the Nambu structure $\{ x, y, z \} = x$ on $N$.
\end{exam}

\section{Gauge transformations and reduction}\label{sec-5}

In this section, we consider the concept of gauge transformation of Nambu-Poisson structures and show that gauge transformation commute with the reduction procedure.

\subsection{Gauge transformations}

Let $(M, \Pi)$ be a Nambu-Poisson manifold of order $n \geq 3$ and take a closed $n$-form $B \in \Omega^n(M).$ Consider the subbundle
$$ \mathcal{T}_B (L_\Pi) := \{ (\Pi^\sharp \alpha, \alpha + i_{\Pi^\sharp \alpha} B) |~ \alpha \in {\Lambda}^{n-1}T^*M \} .$$
Let $\widetilde{B} : TM \rightarrow \Lambda^{n-1}T^*M, ~ X \mapsto i_X B$ be the induced bundle map. If the bundle map
\begin{align}\label{gauge-invertible}
Id + \widetilde{B} \circ \Pi^\sharp : {\Lambda}^{n-1}T^*M \rightarrow {\Lambda}^{n-1}T^*M
\end{align}
is invertible, then
$\mathcal{T}_B (L_\Pi)$ is the graph of a map $\Pi^\sharp (Id + \widetilde{B} \circ \Pi^\sharp)^{-1} : \Lambda^{n-1}T^*M \rightarrow TM$.
Next we will show that the map $\Pi^\sharp (Id + \widetilde{B} \circ \Pi^\sharp)^{-1}$ is skew-symmetric, 
thus, given by a skew-symmetric $n$-tensor field, denoted by $\mathcal{T}_B (\Pi).$
If $\Pi (x) = 0$ for some $x \in M$ then $\Pi_x^\sharp (Id + \widetilde{B} \circ \Pi^\sharp)_x^{-1} : \Lambda^{n-1} T_x^*M \rightarrow T_xM$ is the zero map and hence skew-symmetric. If $\Pi (x) \neq 0$ then there exists a local coordinate $(U ; x^1, \ldots, x^n, x^{n+1}, \ldots, x^m)$ around $x$ such that
$$ \Pi |_U = \frac{\partial}{\partial x^1} \wedge \cdots \wedge \frac{\partial}{\partial x^n}$$
(cf. Theorem \ref{nambu-poisson-tensor-decomposable}). For any locally defined $(n-1)$-form $\alpha$ of the form $\alpha = dx^{i_1} \wedge \cdots \wedge dx^{i_{n-1}}$ with $\{i_1, \ldots, i_{n-1}\} \nsubseteq \{ 1, \ldots, n\}$, we have $\Pi^\sharp \alpha = 0$.
Therefore,
$$ (Id + \widetilde{B} \circ \Pi^\sharp)^{-1} ( dx^{i_1} \wedge \cdots \wedge dx^{i_{n-1}} ) =  dx^{i_1} \wedge \cdots \wedge dx^{i_{n-1}} ,$$
for $\{i_1, \ldots, i_{n-1}\} \nsubseteq \{ 1, \ldots, n\}$. Hence,
$$ \Pi^\sharp (Id + \widetilde{B} \circ \Pi^\sharp)^{-1} ( dx^{i_1} \wedge \cdots \wedge dx^{i_{n-1}} ) = 0 .$$
On the other hand, if $\alpha = dx^1 \wedge \cdots \wedge \widehat{dx^k} \wedge \cdots \wedge dx^n$ for some $k \in \{ 1, \ldots, n\}$, we have
$$ (Id + \widetilde{B} \circ \Pi^\sharp) (dx^1 \wedge \cdots \wedge \widehat{dx^k} \wedge \cdots \wedge dx^n) = dx^1 \wedge \cdots \wedge \widehat{dx^k} \wedge \cdots \wedge dx^n + (-1)^{n-k} i_{ \frac{\partial}{\partial x^k}    } B .$$
Suppose $B$ is locally of the form $B|_U = b ~ dx^1 \wedge \cdots \wedge dx^n + B'$ with $b \in C^\infty(U)$ and $B'$ is an $n$-form defined on
$U$ containing local expressions other than
$dx^1 \wedge \cdots \wedge dx^n$. In that case
 $i_{ \frac{\partial}{\partial x^k}} B =  (-1)^{k-1} b ~ dx^1 \wedge \cdots \wedge \widehat{dx^k} \wedge \cdots \wedge dx^n + B''$, where 
$B''$ is an $(n-1)$-form containing terms $dx^{i_1} \wedge \cdots \wedge dx^{i_{n-1}}$ with $\{i_1, \ldots, i_{n-1}\} \nsubseteq \{ 1, \ldots, n\}$. Therefore,
$$(Id + \widetilde{B} \circ \Pi^\sharp) (dx^1 \wedge \cdots \wedge \widehat{dx^k} \wedge \cdots \wedge dx^n) = [1 + (-1)^{n-1} b]~  dx^1 \wedge \cdots \wedge \widehat{dx^k} \wedge \cdots \wedge dx^n + (-1)^{n-k} B'' .$$
Since $\Pi^\sharp (B'') = 0$ we have
$$ (Id + \widetilde{B} \circ \Pi^\sharp) \bigg( \frac{ dx^1 \wedge \cdots \wedge \widehat{dx^k} \wedge \cdots \wedge dx^n  - (-1)^{n-k} B''   }{[1 + (-1)^{n-1} b]} \bigg) =  dx^1 \wedge \cdots \wedge \widehat{dx^k} \wedge \cdots \wedge dx^n ,$$
which implies that
\begin{align*}
 \Pi^\sharp (Id + \widetilde{B} \circ \Pi^\sharp)^{-1} (dx^1 \wedge \cdots \wedge \widehat{dx^k} \wedge \cdots \wedge dx^n ) &= \frac{1}{[1 + (-1)^{n-1} b]} \Pi^\sharp (dx^1 \wedge \cdots \wedge \widehat{dx^k} \wedge \cdots \wedge dx^n) \\
&= \frac{ (-1)^{n-k}}{ [1 + (-1)^{n-1} b] } \frac{ \partial}{ \partial x^k}.
\end{align*}
Therefore, in any case,
$$ \langle dx^j, \Pi^\sharp (Id + \widetilde{B} \circ \Pi^\sharp)^{-1} (dx^{i_1} \wedge \cdots \wedge dx^{i_{n-1}}) \rangle = - \langle dx^{i_{n-1}}, \Pi^\sharp (Id + \widetilde{B} \circ \Pi^\sharp)^{-1}(dx^{i_1} \wedge \cdots \wedge dx^{i_{n-2}} \wedge dx^j) \rangle$$
holds. This shows that the map $\Pi^\sharp (Id + \widetilde{B} \circ \Pi^\sharp)^{-1}$ is skew-symmetric.
The skew-symmetric $n$-tensor field $\mathcal{T}_B (\Pi) \in \Gamma(\Lambda^n TM)$ is completely determined by
$$ \mathcal{T}_B (\Pi)^\sharp = \Pi^\sharp (Id + \widetilde{B} \circ \Pi^\sharp)^{-1},$$
and, in this case,
$$ \mathcal{T}_B (L_\Pi) = Graph ~( \mathcal{T}_B (\Pi)^\sharp ) = L_{\mathcal{T}_B (\Pi)}.$$
Moreover, $\mathcal{T}_B (L_\Pi)$ is closed under the higher order Dorfman bracket, as
\begin{align*}
& \llbracket (\Pi^\sharp \alpha, \alpha + i_{\Pi^\sharp \alpha} B), (\Pi^\sharp \beta, \beta + i_{\Pi^\sharp \beta} B) \rrbracket \\
=&~ ( [\Pi^\sharp \alpha, \Pi^\sharp \beta] , \mathcal{L}_{\Pi^\sharp \alpha} \beta + \mathcal{L}_{\Pi^\sharp \alpha} i_{\Pi^\sharp \beta} B 
- i_{\Pi^\sharp \beta} d \alpha - i_{\Pi^\sharp \beta} d i_{\Pi^\sharp \alpha} B )\\
=&~ ([\Pi^\sharp \alpha, \Pi^\sharp \beta] , {\textbf \{} \alpha, \beta {\textbf \}}_\Pi + i_{[\Pi^\sharp \alpha, \Pi^\sharp \beta]} B) \qquad (\text{since }B\text{ is closed})\\
=&~ (\Pi^\sharp {\textbf \{} \alpha, \beta {\textbf \}}_\Pi , {\textbf \{} \alpha, \beta {\textbf \}}_\Pi + i_{\Pi^\sharp {\textbf \{} \alpha, \beta {\textbf \}}_\Pi } B ).
\end{align*}
Therefore, it follows from Proposition \ref{nambu-characterization-bi-sheng} that $\mathcal{T}_B (\Pi)$ is a Nambu-Poisson tensor on $M$. 
The Nambu tensor $\mathcal{T}_B (\Pi)$ is called the gauge transformation of $\Pi$ associated with the $n$-form $B$, 
and the Nambu
structures $\Pi, \mathcal{T}_B (\Pi)$ are called gauge equivalent.

\begin{remark}\label{gauge-n-p-same-fol}
 Since the map (\ref{gauge-invertible}) is an isomorphism, it follows that gauge equivalent Nambu-Poisson structures gives rise to same characteristic distribution.
\end{remark}

More generally, gauge equivalent Nambu-Poisson structures corresponds to isomorphic Leibniz algebroids.
\begin{prop}\label{gauge-n-p-same-leib}
 Let $\Pi$ be a Nambu-Poisson structure of order $n$ on $M$, and $\mathcal{T}_B (\Pi)$ be the gauge transformation of $\Pi$ associated with the $n$-form $B$.
Then the Leibniz algebroid structures on $\Lambda^{n-1}T^*M \rightarrow M$ associated to the Nambu-Poisson tensors $\Pi$ and $\mathcal{T}_B (\Pi)$ are isomorphic.
\end{prop}
\begin{proof}
 Consider the bundle isomorphism $(Id + \widetilde{B} \circ \Pi^\sharp):\Lambda^{n-1}T^*M \rightarrow \Lambda^{n-1}T^*M$, given by  $\alpha \mapsto \alpha + i_{\Pi^\sharp \alpha} B$,
for $\alpha \in \Lambda^{n-1}T^*M$. This map commute with the corresponding anchors, as
$$ \Pi^\sharp = \mathcal{T}_B (\Pi)^\sharp \circ (Id + \widetilde{B} \circ \Pi^\sharp) .$$
For any $\alpha, \beta \in  \Lambda^{n-1}T^*M$, we also have
\begin{align*}
& {\textbf \{} (Id + \widetilde{B} \circ \Pi^\sharp) (\alpha) , (Id + \widetilde{B} \circ \Pi^\sharp) (\beta) {\textbf \}}_{\mathcal{T}_B (\Pi)}\\
=&~ \mathcal{L}_{\mathcal{T}_B (\Pi)^\sharp (Id + \widetilde{B} \circ \Pi^\sharp) \alpha} (Id + \widetilde{B} \circ \Pi^\sharp) \beta - i_{\mathcal{T}_B (\Pi)^\sharp (Id + \widetilde{B} \circ \Pi^\sharp) \beta} d ((Id + \widetilde{B} \circ \Pi^\sharp) \alpha)\\
=&~ \mathcal{L}_{\Pi^\sharp \alpha} \beta + \mathcal{L}_{\Pi^\sharp \alpha} i_{\Pi^\sharp \beta} B - i_{\Pi^\sharp \beta} d \alpha - i_{\Pi^\sharp \beta} d i_{\Pi^\sharp \alpha} B \\
=&~ {\textbf \{} \alpha, \beta {\textbf \}}_\Pi + i_{\Pi^\sharp {\textbf \{} \alpha, \beta {\textbf \}}_\Pi} B = (Id + \widetilde{B} \circ \Pi^\sharp) {\textbf \{} \alpha, \beta {\textbf \}}_\Pi .
\end{align*}
Hence the proof.
\end{proof}

\subsection{Gauge transformation commutes with reduction}
\begin{thm}\label{gauge-red=red-gauge}
 Let $(M, \{, \ldots,\})$ be a Nambu-Poisson manifold with associated Nambu-Poisson tensor $\Pi$, $N \subset M$ a submanifold and $i : N \hookrightarrow M$ the inclusion. Let $E \subset TM|_N$ be a subbundle
(may not be canonical) such that $F := E \cap TN$ is a regular, integrable distribution on $N$ and that 
\begin{itemize}
 \item[(i)] if $F_1, \ldots, F_n \in C^\infty(M)_E$ are smooth functions on $M$, then
$$\{F_1 ,\ldots, F_n \} \in C^\infty(M)_F .$$
 \item[(ii)] Moreover,
$$ \Pi^{\sharp} (Ann^1 E) \subseteq TN.$$
\end{itemize}
Let $B$ be a closed $n$-form on $M$ such that the map defined in (\ref{gauge-invertible}) is invertible and that
\begin{itemize}
 \item[a)] $\widetilde{B} : TM \rightarrow \Lambda^{n-1}T^*M$ maps $\widetilde{B} (TN) \subset Ann^{1} E,$
 \item[b)] $B$ projects to an $n$-form $\underline{B} \in \Omega^n(\underline{N})$ on $\underline{N}$.
\end{itemize}
Consider the gauge transformation of $\Pi$ associated with the $n$-form $B$, and let the Nambu structure be denoted by $\mathcal{T}_B(\Pi).$ Then $(M, \mathcal{T}_B(\Pi) )$ reduces to a Nambu-Poisson
manifold $(\underline{N}, \underline{ \mathcal{T}_B(\Pi) })$ which is same as the gauge transformation
of the reduced Nambu-Poisson manifold $(\underline{N}, \underline{\Pi})$ associated with the $n$-form $\underline{B}$, that is, the following diagram
\begin{align*}
\xymatrixrowsep{0.5in}
\xymatrixcolsep{0.7in}
\xymatrix{
(M, \Pi) \ar[r]^{\mathcal{T}_B} \ar[d]_{\pi} & (M, \mathcal{T}_B (\Pi)) \ar[d]^{\underline{\pi}} \\
(\underline{N}, \underline{\Pi}) \ar[r]_{\mathcal{T}_{\underline{B}}} & (\underline{N} , \underline{\mathcal{T}_B (\Pi)} = \mathcal{T}_{\underline{B}} (\underline{\Pi}))
}
\end{align*}
is commutative.
\end{thm}

\begin{proof}
First we will show that the Nambu-Poisson tensor $\mathcal{T}_B (\Pi)$ satisfies the conditions of Theorem \ref{mr-improve-thm}. Since
$\Pi^\sharp (Ann^1 E) \subset TN$ and $\widetilde{B} (TN) \subset Ann^1E$, it follows that, 
$$(Id + \widetilde{B} \circ \Pi^\sharp) : \Lambda^{n-1}T^*M \rightarrow \Lambda^{n-1}T^*M$$ maps 
$(Id + \widetilde{B} \circ \Pi^\sharp)(Ann^1E) \subset Ann^1E.$ As $(Id + \widetilde{B} \circ \Pi^\sharp)$ is an isomorphism,
$(Id + \widetilde{B} \circ \Pi^\sharp)(Ann^1E) = Ann^1E.$ Therefore, for any $\eta \in Ann^1 E$,
$$ \mathcal{T}_B (\Pi)^\sharp (\eta) = \Pi^\sharp (Id + \widetilde{B} \circ \Pi^\sharp)^{-1} (\eta) \in \Pi^\sharp (Ann^1E) \subseteq TN.$$
Moreover, for any $F_1, \ldots, F_n \in C^\infty(M)_E$, we have
\begin{align*}
 \{F_1, \ldots, F_n \}_{\mathcal{T}_B (\Pi)} =&~ \langle \mathcal{T}_B (\Pi)^\sharp (dF_1 \wedge \cdots \wedge dF_{n-1}) , dF_n \rangle \\
=&~ \langle  \Pi^\sharp (Id + \widetilde{B} \circ \Pi^\sharp)^{-1}  (dF_1 \wedge \cdots \wedge dF_{n-1}) , dF_n \rangle .
\end{align*}
Therefore, for any $V \in \Gamma (TM)$ with $V|_N \in F$, we have
\begin{align*}
 &\langle d (\{F_1, \ldots, F_n \}_{\mathcal{T}_B (\Pi)} ) , V \rangle |_N\\ =&~ ( \mathcal{L}_{\Pi^\sharp (Id + \widetilde{B} \circ \Pi^\sharp)^{-1}  (dF_1 \wedge \cdots \wedge dF_{n-1})} dF_n )(V)|_N \\
=&~ \big[ \mathcal{L}_{\Pi^\sharp (Id + \widetilde{B} \circ \Pi^\sharp)^{-1}  (dF_1 \wedge \cdots \wedge dF_{n-1})} V (F_n) - dF_n ([\Pi^\sharp (Id + \widetilde{B} \circ \Pi^\sharp)^{-1}  (dF_1 \wedge \cdots \wedge dF_{n-1}) , V]) \big]|_N = 0.
\end{align*}
This shows that, $\{ F_1, \ldots, F_n \}_{\mathcal{T}_B (\Pi)} \in C^\infty(M)_F.$ Thus, the Nambu structure $\mathcal{T}_B (\Pi)$ is reducible by Theorem
\ref{mr-improve-thm}.

Moreover, since $B$ is closed, the $n$-form $\underline{B} \in \Omega^n(\underline{N})$ on $\underline{N}$ is closed. The bundle map
$$ Id + \widetilde{\underline{B}} \circ \underline{\Pi}^\sharp : \Lambda^{n-1} T^*\underline{N} \rightarrow \Lambda^{n-1} T^*\underline{N}$$
is invertible and the inverse is given as follows. For any $f_1, \ldots, f_n \in C^\infty(\underline{N})$, let 
$F_1, \ldots, F_n \in C^\infty(M)_E$ be arbitrary extensions of $f_1 \circ \pi, \ldots, f_n \circ \pi$
fron $N$ to $M$. If the inverse $(Id + \widetilde{B} \circ \Pi^\sharp)^{-1}(dF_1 \wedge \cdots \wedge dF_{n-1})$ is locally given by the sum
$\sum_{j_1, \ldots, j_{n-1}}  H_{j_1 \ldots j_{n-1}} dH_{j_1} \wedge \cdots \wedge dH_{j_{n-1}}$, for some locally defined functions $H_j$'s on $M$ with differentials vanishing on $E$, then the inverse
$(Id + \widetilde{\underline{B}} \circ \underline{\Pi}^\sharp)^{-1} (df_1 \wedge \cdots \wedge df_{n-1})$ is locally given by by the sum
$\sum_{j_1, \ldots, j_{n-1}}  h_{j_1 \ldots j_{n-1}} dh_{j_1} \wedge \cdots \wedge dh_{j_{n-1}}$, where $h_j$'s are restriction of $H_j$'s on $N$.
 
From the reducibility of Nambu structures $\Pi$ and $\mathcal{T}_B (\Pi)$, we have
\begin{align}
\{ f_1, \ldots, f_n \}_{\underline{\Pi}} \circ \pi &= \{F_1, \ldots, F_n \} \circ i, \label{eqn1}\\
 \{ f_1, \ldots, f_n \}_{\underline{\mathcal{T}_B \Pi}} \circ \pi &= \{F_1, \ldots, F_n \}_{\mathcal{T}_B \Pi} \circ i . \label{eqn2}
\end{align}

Therefore, for any $x \in N$,
\begin{align*}
&\{ f_1, \ldots, f_n \}_{\mathcal{T}_{\underline{B}} (\underline{\Pi})} (\pi(x))\\
=&~ \Big\langle  \mathcal{T}_{\underline{B}} (\underline{\Pi})^\sharp (df_1 \wedge \cdots \wedge df_{n-1}) , df_n \Big\rangle (\pi(x)) \\
=&~ \Big\langle \underline{\Pi}^\sharp (Id + \widetilde{\underline{B}} \circ \underline{\Pi}^\sharp)^{-1} (df_1 \wedge \cdots \wedge df_{n-1}) , df_n \Big\rangle (\pi(x)) \\
=&~ \sum_{j_1, \ldots, j_{n-1}} \langle h_{j_1 \ldots j_{n-1}} \underline{\Pi}^\sharp (dh_{j_1} \wedge \cdots \wedge dh_{j_{n-1}}), df_n \rangle (\pi(x)) \\
=&~ \sum_{j_1, \ldots, j_{n-1}} ( h_{j_1 \ldots j_{n-1}} \{ h_{j_1}, \ldots, h_{j_{n-1}}, f_n \}_{\underline{\Pi}}) (\pi(x))\\
=&~ \sum_{j_1, \ldots, j_{n-1}} ( H_{j_1 \ldots j_{n-1}} \{ H_{j_1}, \ldots, H_{j_{n-1}}, F_n \}) (i(x)) \qquad (\text{by Equation } (\ref{eqn1}))\\
=&~ \sum_{j_1, \ldots, j_{n-1}} \langle H_{j_1 \ldots j_{n-1}} \Pi^\sharp (dH_{j_1} \wedge \cdots \wedge dH_{j_{n-1}}), dF_n \rangle (i(x)) \\
=&~ \Big \langle \Pi^\sharp (Id + \widetilde{B} \circ \Pi^\sharp)^{-1}(dF_1 \wedge \cdots \wedge dF_{n-1}), dF_n \Big \rangle (i(x))\\
=&~ \Big\langle \mathcal{T}_B (\Pi)^\sharp (dF_1 \wedge \cdots \wedge dF_{n-1}), dF_n \Big \rangle (i(x)) = \{F_1, \ldots, F_{n-1}, F_n \}_{\mathcal{T}_B (\Pi)} (i(x)).
\end{align*}
Thus by Equation (\ref{eqn2}), it follows that
$$\{f_1, \ldots, f_n \}_{\underline{\mathcal{T}_B \Pi}} \circ \pi = \{ f_1, \ldots, f_n \}_{\mathcal{T}_{\underline{B}} (\underline{\Pi})} \circ \pi.$$
Since $\pi$ is surjective, we have $\underline{\mathcal{T}_B \Pi} = \mathcal{T}_{\underline{B}} (\underline{\Pi}).$ Hence the proof.\\
\end{proof}

\noindent {\bf Acknowledgment.}
The author wish to thank Prof. Goutam Mukherjee for his carefully reading the manuscript.
The author would also like to thank the referee for his comments and suggestions on the earlier version of the paper that have improved the exposition.



\end{document}